\documentclass[preprint,11pt]{elsarticle}

\usepackage{amssymb,amsmath,amsthm}
\usepackage{graphicx}
\usepackage{float}
\restylefloat{figure}

\newcommand{\beq}[1]{ \begin{equation}\label{#1} }
\newcommand{\eeq}{\end{equation}}
\numberwithin{equation}{section}
\DeclareMathOperator*{\esssup}{ess\,sup}

\usepackage{fullpage}
\newtheorem{theorem}{Theorem}
\newtheorem{lemma}{Lemma}
\newtheorem{corollary}{Corollary}
\newtheorem{example}{Example}

\newtheorem{proposition}{Proposition}
\journal{Journal of Mathematical Analysis and Applications}

\begin{document}

\begin{frontmatter}

\title{Asymptotic properties of neutral type linear systems}

\author[label1]{Leonid Berezansky}
\author[label2]{Elena Braverman}
\address[label1]{Dept. of Math.,
Ben-Gurion University of the Negev,
Beer-Sheva 84105, Israel}
\address[label2]{Dept. of Math. and Stats., University of
Calgary, 
Calgary, AB  T2N 1N4, Canada}

\begin{abstract}
Exponential stability and solution estimates are investigated for a delay system 
$$
\dot{x}(t) - A(t)\dot{x}(g(t))=\sum_{k=1}^m B_k(t)x(h_k(t))
$$
of a neutral type,
where $A$ and $B_k$ are $n\times n$ bounded matrix functions, and $g, h_k$ are delayed arguments.
Stability tests are applicable to a wide class of linear neutral systems with time-varying coefficients and delays.
In addition, explicit exponential estimates for solutions of both homogeneous and non-homogeneous neutral systems 
are obtained for the first time. 
These inequalities are not just asymptotic estimates, they are valid on every finite segment and evaluate both short- and long-term
behaviour of solutions. 
\end{abstract}



\end{frontmatter}

\section{Introduction}

There are many papers and monographs on stability of scalar neutral differential equations, see a review of
explicit stability tests for this class of equations in \cite{BB1, BB2}.
For linear vector neutral delay equations, explicit exponential stability results can be found 
in the monographs \cite{AzbSim,Gil,Gop2,KM,KolmNos,Shaikhet} and the papers 
\cite{Walther,Dom,Gop1, GH, Han,Li-Ming,Ngoc2,Ngoc1,Park}.

A  most popular approach to study stability for all classes 
of functional differential equations, including neutral,
is the application of Lyapunov-Krasovskii functionals \cite{Gop2, KM, KolmNos, Shaikhet}.
LMI (linear matrix inequality) method is also based on the method of Lyapunov-Krasovskii functionals \cite{Fridman,Liu, Park}.
Some stability results were obtained by using the Bohl-Perron theorem \cite{AzbSim,BB1, BB2, BDSZ1, BDSZ2, BDSZ3, Dom} and application of special properties of matrices, including fundamental matrix functions for differential equations \cite{Gil, GP, GH, Ngoc2, Ngoc1}. 
Several tests 
for linear and nonlinear scalar neutral equations were obtained by the fixed point method \cite{Jin, Zhao}.
Stability conditions based on the application of fixed point theorems usually include the first and the second derivatives of delay functions.
A more detailed discussion of these methods and comparison of stability 
results of the present paper to some earlier obtained is postponed to the end of the paper.

Recently in \cite{BB2}, exponential estimates were derived for a neutral delay equation
\begin{equation}\label{1.1}
\dot{x}(t) - a(t)\dot{x}(g(t)) =\sum_{k=1}^m b_k(t)x(h_k(t)), ~~t \geq t_0\geq 0
\end{equation}
in the scalar case.
In the present paper we extend the results obtained in \cite{BB2} to a system
\begin{equation}\label{1.2}
\dot{x}(t) - A(t)\dot{x}(g(t)) =\sum_{k=1}^m B_k(t)x(h_k(t)), ~~t \geq t_0\geq 0,
\end{equation}
where some of the stability tests are new also for scalar equation \eqref{1.1} (we will discuss this question in the last section)
and for a vector linear delay differential equations without the neutral part when $A(t)\equiv 0$.

We study non-autonomous system (\ref{1.2}) and get stability tests where not only the maximal of all the delays 
but each $h_k$ is taken into account.

These  tests are explicit, we compare them to some earlier results. 
Our second goal is to develop exponential estimates for solutions of a non-homogeneous version of (\ref{1.2}), 
with coefficient of the exponential part described by a maximum of the initial functions, while the maximum of the right-hand side leads to a constant bound, not an exponential estimate of a solution.
We evaluate short-term, as well as long-term (asymptotic) behaviour of a solution. 
To the best of our knowledge, for neutral differential systems such exponential  estimates have not been known, and the present paper fills the gap.

Equation \eqref{1.2} includes equations without the neutral part
$$
\dot{x}(t)=\sum_{k=1}^m B_k(t)x(h_k(t)), ~~t \geq t_0\geq 0,
$$
an equation with one delay ($m=1$) and autonomous equations.

In the monograph~\cite{KM}, the method of Lyapunov-Krasovskii functionals was applied to
local asymptotic stability for vector nonlinear delay differential equations.
Relevant Theorem~3.1 in~\cite[p. 286]{KM} was
actually obtained for nonlinear equations, and we reformulate it for the linear case
\begin{equation}\label{4.0}
\dot{x}(t)=B(t)x(t-h),\,\,\,t\geq 0,
\end{equation}
where $B$ is an $n\times n$ matrix with locally essentially bounded entries, and $h>0$ is a constant delay.

Below, $\|\cdot\|$ is a norm in ${\mathbb R}^n$, and $\mu(A)$ is a matrix measure which is defined in the next section.

\begin{proposition}\label{p3}
Let $t_0\geq 0$, $\mu(B(t))\leq -\beta<0$ and
$\displaystyle 
h \, \sup_{t\geq t_0}\|B(t)\|^2< 
\inf_{t\geq t_0}\,|\mu(B(t))|$.

Then equation~\eqref{4.0} is uniformly asymptotically stable.
\end{proposition}

A significant part of known asymptotic stability tests for neutral equations was obtained for autonomous systems with a non-delay term, such as 
\begin{equation}\label{1.2a}
\dot{x}(t) - A_1 \dot{x}(t-h_1) = A_0 x(t)+A_2 x(t-h_2), ~h_1, h_2>0, ~t \geq t_0.
\end{equation}

The following stability tests were obtained using Lyapunov-Krasovskii functionals.

\begin{proposition} \cite[P. 367]{KM},\cite{Li-Ming}
\label{p1}
Let
$
\|A_1\|<1$ and $\displaystyle \mu (A_0)+\frac{\|A_0\|+\|A_1\|\|A_2\|}{1-\|A_1\|}<0$.

Then system \eqref{1.2a} is asymptotically stable.
\end{proposition}

\begin{proposition}\label{p2} \cite[P. 369]{KM} 
If $\|A_1\|+ h_2\|A_2\|<1$ and $$\displaystyle \mu (A_0+A_2)+\frac{\|A_0+A_2\|(\|A_1\|+h_2\|A_2\|)}{1-\|A_1\|- h_2\|A_2\|}<0$$
then system~\eqref{1.2a} is asymptotically stable.
\end{proposition}

The structure of the paper is the following. Section~2 presents preliminaries, while
main results on solution estimates and exponential stability are obtained in Section 3.
We get solution estimates for system \eqref{1.2} or its non-homogeneous counterpart and derive stability tests for \eqref{1.2}.
Section~4 illustrates novelty of the results with examples and indicates possible directions of further research. We also discuss Propositions \ref{p3}-\ref{p2} and compare them to the results obtained in the paper in Section~4.

\section{Preliminaries}

There exists a well developed theory of neutral linear differential equations including existence, uniqueness, representation of solutions,
see e.g. \cite{AzbSim,Gil,Hale,KM,KolmNos}. We use some results and definitions from these monographs, including definitions
of asymptotic and uniform exponential stability.

Let  $\|x\|$ be an arbitrary vector norm of $x \in {\mathbb R}^n$, the same notation is  
used for the induced matrix norm. 
The matrix measure 
(often referred to as the logarithmic matrix norm) of an $n\times n$ matrix $C$
is defined as
$$
\mu (C)= \left. \left. \lim_{\nu\to 0^+}{\nu}^{-1}\right( \,{\|E+\nu C\|-1\,}\right),
$$
where $E$ is the $n\times n$ identity matrix. Some of its properties are useful in calculations:
$$
|\mu(C)|\leq \|C\|, ~~\mu(C_1+C_2)\leq \mu(C_1)+\mu(C_2), ~~\mu(\lambda C)=\lambda \mu(C), ~\lambda >0.
$$
For details, we refer readers to~\cite{S,Z}
and~\cite[Table 3.1, p. 286]{KM}.

In particular, if we use the maximum norm $\|x\|_{\infty}=\max \{|x_1|,\dots,|x_n|\}$ then 
$$\|C\|=\max_{1\leq i \leq n} \left\{ \sum_{j=1}^n |c_{ij}| \right\}, 
~~ \mu (C)=\max_{1\leq i \leq n} \left\{  c_{ii}+ \sum_{j\neq i} |c_{ij}|  \right\}.
$$
For either a fixed interval $J=[t_0,t_1]$ or  $J=[t_0,\infty)$, consider the space $L_{\infty}(J)$ 
with an essential supremum norm
$\|x\|_J= \esssup_{t\in J} \|x(t)\|$.
Further, by $\sup$ of any Lebesgue measurable function we mean an essential supremum without mentioning it.

For system  \eqref{1.2}, the functions $A, B_k, g, h_k$ are assumed to be Lebesgue measurable on 
$[t_0,\infty)$, where $t_0\geq 0$ is an arbitrary initial point,  
$n\times n$ matrix functions $A$ and $B_k$ are essentially bounded
on $[t_0,\infty)$, delays in $g$ and $h_k$ are bounded: $0\leq t-g(t)\leq \sigma$, $0\leq t-h_k(t)\leq \tau_k$ for $\sigma>0$, $\tau_k>0$, $k=1, \dots,m$. 
The coefficient $A$ satisfies $\|A(t)\|\leq a_0 <1$ for some $a_0<1$ for $t\geq t_0$.

The condition on the neutral delay that, once the Lebesgue
measure of the set $U$ is zero, the set $g^{-1}(U)$ has also the zero measure,
ensures that $x(g(t))$ is properly defined and  measurable, once $x$ is measurable.

In addition to  (\ref{1.2}), we introduce a system 
\begin{equation}
\label{2.1}
\dot{x}(t) - A(t)\dot{x}(g(t)) =\sum_{k=1}^m B_k(t)x(h_k(t)) + f(t), ~t\geq t_0
\end{equation}
with initial conditions
\begin{equation}
\label{2.2}
x(t)=\Phi (t) \mbox{~for~~} t \leq t_0, \quad \dot{x}(t)=\Psi (t) \mbox{~for~~}t<t_0.
\end{equation}
Here $f \in L_{\infty}([t_0,c])$ for any $c>t_0$, while 
$\Phi:[t_0-\max_k \tau_k,t_0] \rightarrow {\mathbb R}^n$ and $\Psi : [t_0-\sigma,t_0) \rightarrow {\mathbb R}^n$, 
in addition to boundedness, are supposed to be Borel measurable.

Everywhere below we let these assumptions be satisfied, for example, considering (\ref{2.1}),(\ref{2.2}) or
any delay system, without repeating them.

By {\em a solution of problem} (\ref{2.1}),(\ref{2.2}) we mean a locally absolutely continuous on $[t_0,\infty)$
function $x: {\mathbb R} \rightarrow {\mathbb R}^n$ 
satisfying system (\ref{2.1}) for almost all $t\in [t_0,\infty)$ and
initial conditions (\ref{2.2}) when $t\leq t_0$.

The above conditions on the parameters of the equations and initial conditions  guarantee existence and uniqueness of a solution of 
initial value problem \eqref{2.1},\eqref{2.2} \cite{AzbSim}, and a solution representation which will later be applied.

For each $s\geq t_0$ we introduce {\em the fundamental matrix} as a solution $X(t,s)$ of 
\begin{equation*}
\dot{x}(t) - A(t)\dot{x}(g(t)) =\sum_{k=1}^m B_k(t)x(h_k(t)), ~x(t)=0,~\dot{x}(t)=0,~t<s,~x(s)=E,
\end{equation*}
where $E$ is the identity matrix. 
By definition, $X(t,s)$ vanishes whenever $t<s$.

If there exist positive numbers 
$M$ and $\gamma$ for which
any solution of (\ref{1.2}),(\ref{2.2}) 
satisfies
$$
\|x(t)\|\leq M e^{-\gamma (t-t_0)} \left[ \sup_{t \in [t_0 - \max_{1 \leq k \leq m} \tau_k,  t_0]} \|\Phi(t)\|+
\sup_{t \in [t_0 - \sigma,t_0]} \|\Psi(t)\| \right], ~~t\geq t_0,
$$
where neither $M$ nor $\gamma$ depends on $t_0 \geq 0$ and the initial functions $\Phi$ and $\Psi$, 
system (\ref{1.2}) is called {\em uniformly exponentially stable}. 
Also, exponential estimates of a fundamental matrix $X(t,s)$ will be considered 
$\displaystyle \|X(t,s)\|\leq M_0 e^{-\gamma_0(t-s)}$ for $t\geq s\geq t_0$ and some 
$M_0>0$ and $\gamma_0>0$ which hold with $M_0=M$, $\gamma_0=\gamma$ if (\ref{1.2}) is uniformly exponentially stable. 

Define a linear bounded operator 
on the space $L_{\infty}[t_0,t_1]$ as 
$$
\displaystyle 
(Sy)(t)=\left\{\begin{array}{ll}
A(t)y(g(t)),& g(t)\geq t_0,\\
0,& g(t)<t_0.\\
\end{array}\right. 
$$
Note that there exists a unique
solution (see, for example, \cite{AzbSim}) of
problem (\ref{2.1}),(\ref{2.2}).
This solution 
has the representation
\begin{equation}
\label{2.4}
\begin{array}{ll}
\displaystyle 
x(t)  = & \displaystyle  X(t,t_0)x_0+ \!\! \int\limits_{t_0}^t \!\!
X(t,s)\left[(I-S)^{-1}f \right](s)ds 
\\ & \displaystyle
+ \!\! \int\limits_{t_0}^{t_0+\sigma} \!\!\!\!
X(t,s)\left[ (I-S)^{-1}(A(\cdot)\Psi(g(\cdot))) \right](s)ds 
\\ & \displaystyle 
+ \sum_{k=1}^m \int\limits_{t_0}^{t_0+\tau_k} 
X(t,s)\left[(I-S)^{-1}(B_k(\cdot)\Phi(h_k(\cdot)))\right](s)ds,
\end{array}
\end{equation}
where $I:L_{\infty}(J) \to L_{\infty}(J)$ is the identity operator, and we set
$\Psi(g(t))=0$ for $g(t)\geq t_0$ and
$\Phi(h_k(t))=0$ whenever $h_k(t)\geq t_0 $. 
For any $t_1>t_0$, the norm of $(I-S)^{-1}$ satisfies \cite{AzbSim}
\begin{equation}
\label{norm}
\|(I-S)^{-1}\|_{L_{\infty}[t_0,t_1] \to L_{\infty}[t_0,t_1]}\leq \frac{1}{1-\|A\|_{[t_0,\infty)}} \, .
\end{equation}

We use an auxiliary system with a term not involving any delay
\begin{equation}\label{2.6}
\dot{z}(t)-A_0(t)\dot{z}(g(t))=C(t)z(t) 
+ \sum_{k=0}^m D_k(t)z(h_k(t)), ~t \geq t_0.
\end{equation}
First, let us evaluate the fundamental matrix of system \eqref{2.6}. 
Denote 
\begin{equation}
\label{D_def}
D(t):=\sum_{k=0}^m D_k(t).
\end{equation}

\begin{lemma}\label{lemma2.1}
Let  
$\|A_0\|_{[t_0,\infty)}<1$, there exist an $\alpha_0 <0$ such that $\mu (C(t)+D(t))\leq \alpha_0$, 
$Z$ be a fundamental matrix of system \eqref{2.6}, and
\begin{equation}\label{2.7}
K_0:= 
\frac{\|C\|_{[t_0,\infty)}+\sum_{k=0}^m \|D_k\|_{[t_0,\infty)}}{1-\|A_0\|_{[t_0,\infty)}} 
\left(\left\|\frac{A_0}{\mu(C+D)}\right\|_{[t_0,\infty)}+\sum_{k=0}^m  \tau_k \left\|\frac{D_k}{\mu(C+D)}\right\|_{[t_0,\infty)}\right)<1.
\end{equation}
Then,  $Z(t,s)$ satisfies 
\begin{equation}\label{2.5}
\|Z(t,s)\|\leq K, ~~t \geq s \geq t_0,
\end{equation}
where the bound $K$ is $K= (1-K_0)^{-1}$.
\end{lemma}

\begin{proof}
We denote $z(t)=Z(t,t_0)$ to make expressions shorter. 
By definition, $z$ satisfies (\ref{2.6}), where the initial matrix is $z(t_0)=E$, and the initial functions are identically equal to the zero matrix.
For an arbitrary $t_1>t_0$, denote $J=[t_0,t_1]$. By (\ref{norm}) and (\ref{2.6}), we get the inequality
\begin{equation}\label{2.8}
 \|\dot{z}\|_J\leq  \frac{\|C\|_{[t_0,\infty)}+\sum_{k=0}^m \|D_k\|_{[t_0,\infty)}}{1-\|A_0\|_{[t_0,\infty)}} \|z\|_J.
\end{equation}
A fundamental matrix $Y(t,s)$ of the ordinary differential equation
$$
\dot{y}(t)=(C(t)+D(t))y(t)
$$
has an exponential estimate \cite[Page 9]{Coppel}
\begin{equation}\label{2.9}
\|Y(t,s)\|\leq e^{\int_s^t \mu (C(\xi)+D(\xi))d\xi}.
\end{equation}
Further, we rewrite (\ref{2.6}) with a different non-delay term 
\begin{equation}\label{2.9a}
\dot{z}(t)=[C(t)+D(t)]z(t)+A_0(t)\dot{z}(g(t))-\sum_{k=0}^m D_k(t)\int_{h_k(t)}^t \dot{z}(\xi)d\xi.
\end{equation}

Integrating from $t_0$ to $t \leq t_1$, we get: 
$$
\begin{array}{ll}
z(t)= & \displaystyle Y(t,t_0)+\int_{t_0}^t Y(t,s)\mu (C(s)+D(s))
\left [\frac{A_0(s)}{\mu(C(s)+D(s))}\dot{z}(g(s)) \right. \\ 
& \left. \displaystyle  -\sum_{k=0}^m \frac{D_k(s)}{\mu(C(s)+D(s))}\int_{h_k(s)}^s \dot{z}(\xi)d\xi\right]ds.
\end{array}
$$
Therefore, using (\ref{2.8}),(\ref{2.9}), inequality $\mu (C(t)+D(t))\leq \alpha_0<0$  
leads to
$$
\left\| \int_{t_0}^t Y(t,s) \mu (C(s)+D(s))~ds \right\| \leq 1
$$
and also recalling  $K_0$ from (\ref{2.7}),
$$
\|z\|_J\leq 1+\left(\left\|\frac{A_0}{\mu(C+D)}\right\|_{[t_0,\infty)}
+\sum_{k=0}^m \tau_k\left\|\frac{D_k}{\mu(C+D)}\right\|_{[t_0,\infty)}\right)\|\dot{z}\|_J
\leq 1+ K_0 \| z \|_J.
$$
We get
$ \|Z(t,t_0)\|_{[t_0,t_1]} \leq (1-K_0)^{-1}$, and the expression in 
the right-hand side 
does not depend on $t_1$,
which implies
 $ \displaystyle \|Z(t,t_0)\|_{[t_0,\infty)}\leq  (1-K_0)^{-1}$.
This inequality is also valid if $t_0$ is replaced with $s>t_0$.
Hence estimate (\ref{2.5}) is valid, which concludes the proof.
\end{proof}

\begin{lemma}\label{lemma2.2}
Let  $\displaystyle \|A_0\|_{[t_0,\infty)}+\sum_{k=0}^m  \tau_k \|D_k\|_{[t_0,\infty)}<1$, $\mu (C(t)+D(t))\leq \alpha_0$ for some $\alpha_0 <0$, where $D$ from (\ref{D_def}) is used, $Z$ be a fundamental matrix of (\ref{2.6}), and the constant
\begin{equation*}
\displaystyle L_0:=  \displaystyle  \frac{\|C+D\|_{[t_0,\infty)}}{1-\|A_0\|_{[t_0,\infty)}-\sum_{k=0}^m  \tau_k \|D_k\|_{[t_0,\infty)}}   
\displaystyle \left(\left\|\frac{A_0}{\mu(C+D)}\right\|_{[t_0,\infty)}+\sum_{k=0}^m  \tau_k \left\|\frac{D_k}{\mu(C+D)}\right\|_{[t_0,\infty)}\right)<1. 
\end{equation*}
Then, $Z(t,s)$  satisfies (\ref{2.5}), where $L= (1-L_0)^{-1}$ is used instead of $K$.
\end{lemma}
\begin{proof}
Again, $z(t)=Z(t,t_0)$ satisfies (\ref{2.6}) with the zero initial matrix-functions and $z(t_0)=E$.
Let $J=[t_0,t_1]$ for any $t_1>t_0$. Equality \eqref{2.9a} implies
$$
\|\dot{z}\|_J\leq \|C+D\|_{[t_0,\infty)}\|z\|_J+ \left( \|A_0\|_{[t_0,\infty)}+\sum_{k=0}^m\tau_k \|D_k\|_{[t_0,\infty)} \right) \|\dot{z}\|_J,
$$
therefore
$$
\|\dot{z}\|_J\leq \frac{\|C+D\|_{[t_0,\infty)}}{1-\|A_0\|_{[t_0,\infty)}-\sum_{k=0}^m\tau_k \|D_k\|_{[t_0,\infty)}}\|z\|_J.
$$
The rest of the proof repeats the scheme for Lemma~\ref{lemma2.1} and thus is omitted.
\end{proof}


\section{Main Results}

\subsection{Boundedness and solution estimates}

For any $\lambda \in {\mathbb R}$, denote the  matrix function
\begin{equation}\label{3.0}
P(t):=\sum_{k=1}^m e^{\lambda(t-h_k(t))}B_k(t)-\lambda e^{\lambda(t-g(t))} A(t)+\lambda E.
\end{equation}
We will show that, once the matrix measure of $P(t)$ in \eqref{3.0} is less than a constant negative number, for a positive $\lambda$ and under some other natural assumptions, \eqref{1.2} is globally exponentially stable,
and a solution estimate can be derived.

\begin{theorem}\label{theorem3.1}
Let $\lambda$ and $\beta$ be positive constants for which
\begin{equation}\label{3.1}
\mu (P(t))\leq -\beta, \quad t \geq t_0, \quad e^{\lambda\sigma}\|A\|_{[t_0,\infty)}<1,
\end{equation}
\begin{equation}
\label{3.2}
\begin{array}{ll} 
M_1:=
& \displaystyle  
\frac{\lambda +\sum\limits_{k=1}^m e^{\lambda\tau_k}\|B_k\|_{[t_0,\infty)}+\lambda e^{\lambda\sigma}\|A\|_{[t_0,\infty)}}{1-e^{\lambda\sigma}\|A\|_{[t_0,\infty)}} 
\vspace{2mm} \\ 
& \displaystyle  \times \left(\left\|\frac{A}{\mu (P)}\right\|_{[t_0,\infty)} \!\!\! \!\!\! \!\!\!\! (1+\lambda\sigma)e^{\lambda\sigma}
+\sum_{k=1}^m  \left\|\frac{ B_k}{\mu (P)}\right\|_{[t_0,\infty)} \!\!\!\!\!\!\!\!\!\! e^{\lambda\tau_k}\tau_k \right)<1,
\end{array}
\end{equation}
where $P$ is defined in \eqref{3.0}.
Then, for $M_0:=(1-M_1)^{-1}$, the solution of (\ref{2.1}),(\ref{2.2}) satisfies
\begin{equation}
\label{3.3}
\begin{array}{ll} \| x(t) \| \leq & \displaystyle M_0e^{-\lambda(t-t_0)}\left[  \| x(t_0)  \|
+\frac{e^{\lambda \sigma}-1}{\lambda(1-\|A\|_{[t_0,\infty)})} \|A\|_{[t_0,\infty)}\|\Psi\|_{[t_0-\sigma,t_0]}\right.
\vspace{2mm}
\\
& \displaystyle \left.+\sum_{k=1}^m\frac{e^{\lambda \tau_k}-1}{\lambda(1-\|A\|_{[t_0,\infty)})}
 \|B_k\|_{[t_0,\infty)}\|\Phi\|_{[t_0-\tau_k,t_0]}\right]+\frac{M_0}{\lambda(1-\|A\|_{[t_0,\infty)})}\|f\|_{[t_0,t]}. \end{array}
\end{equation}
\end{theorem}

\begin{proof}
We start with homogeneous system (\ref{1.2}).
Substituting $x(t)=e^{-\lambda(t-t_0)}y(t)$ into (\ref{1.2}), we obtain 
\begin{equation}\label{3.4}
\dot{y}(t)-A(t)e^{\lambda(t-g(t))}\dot{y}(g(t))=\lambda y(t)-\lambda e^{\lambda(t-g(t))}A(t)y(g(t))
+\sum_{k=1}^m e^{\lambda(t-h_k(t))}B_k(t)y(h_k(t)).
\end{equation}
Equation (\ref{3.4}) has the form of (\ref{2.6}) with 
$$
A_0(t)=e^{\lambda(t-g(t))} A(t), ~~C(t)=\lambda E, ~~D_0(t)=-\lambda e^{\lambda(t-g(t))} A(t),  ~~h_0(t)=g(t),
$$$$
D_k(t)=e^{\lambda(t-h_k(t))}B_k(t),~k=1,\dots,m,
~~ D(t)=\sum_{k=0}^m D_k(t). 
$$
Therefore $P(t)=C(t)+D(t)$,
$$
\left\|\frac{A_0}{\mu (C+D)}\right\|_{[t_0,\infty)}\leq e^{\lambda\sigma}\left\|\frac{A}{\mu (P)}\right\|_{[t_0,\infty)},
\left\|\frac{D_0}{\mu (C+D)}\right\|_{[t_0,\infty)}\leq \lambda e^{\lambda\sigma}\left\|\frac{A}{\mu (P)}\right\|_{[t_0,\infty)},
$$ $$
\left\|\frac{D_k}{\mu (C+D)}\right\|_{[t_0,\infty)}\leq e^{\lambda\tau_k}\left\|\frac{B_k}{\mu (P)}\right\|_{[t_0,\infty)}, k=1,\dots,m,
$$$$
\frac{\|C\|_{[t_0,\infty)}+\sum_{k=0}^m \|D_k\|_{[t_0,\infty)}}{1-\|A_0\|_{[t_0,\infty)}}\leq
\frac{\lambda+\sum_{k=1}^m e^{\lambda\tau_k}\|B_k\|_{[t_0,\infty)}+\lambda e^{\lambda\sigma}\|A\|_{[t_0,\infty)}}{1-e^{\lambda\sigma}\|A\|_{[t_0,\infty)}}.
$$
Inequalities (\ref{3.1}) and (\ref{3.2}) imply  the assumptions of Lemma~\ref{lemma2.1}, in particular, (\ref{2.7}).  
Let $Y(t,s)$ be a fundamental matrix of (\ref{3.4}), then we can apply Lemma~\ref{lemma2.1} to deduce $\|Y(t,s)\|\leq M_0$. Once $X(t,s)$ is 
a fundamental matrix  of system (\ref{1.2}), it satisfies $X(t,s)=e^{-\lambda(t-s)}Y(t,s)$.
This implies the exponential estimate $\|X(t,s)\| \leq M_0 e^{-\lambda(t-s)}$. 
Let $x$ be a solution of problem (\ref{1.2}),(\ref{2.2}).
We use solution representation  (\ref{2.4}) and inequality (\ref{norm})
\begin{align*}
 \|x(t)\|\leq & \|X(t,t_0)\|\|x_0\|
+  \left. \left. \int_{t_0}^{t_0+\sigma} \|X(t,s)\| \right\| (I-S)^{-1}  \right\|_{L_{\infty}[t_0,t_1] \to L_{\infty}[t_0,t_1]}  \|A(s)\|\|\Psi(g(s))\|ds
\\
& +\sum_{k=1}^m \left. \left. \int_{t_0}^{t_0+\tau_k} \|X(t,s)\| \right\| (I-S)^{-1}  \right\|_{L_{\infty}[t_0,t_1] \to L_{\infty}[t_0,t_1]}  \|B_k(s)\|\|\Phi(h_k(s))\|ds
\\
\leq  & M_0 e^{-\lambda(t-t_0)} \|x(t_0)\|+\frac{M_0}{\lambda\left(1-\|A\|_{[t_0,\infty)}\right)} 
 \|A\|_{[t_0,\infty)} \left( e^{-\lambda (t-t_0-\sigma)}-e^{-\lambda(t-t_0)} \right)\|\Psi\|_{[t_0-\sigma,t_0]}
\\
& +\frac{M_0}{\lambda\left(1-\|A\|_{[t_0,\infty)}\right)} 
\sum_{k=1}^m \|B_k\|_{[t_0,\infty)} \left( e^{-\lambda (t-t_0-\tau_k)}
-e^{-\lambda(t-t_0)} \right)\|\Phi\|_{[t_0-\tau_k,t_0]},
\end{align*}
which immediately yields (\ref{3.3}) when $f \equiv 0$.

For any $f$ in (\ref{2.1}), we employ the above inequality for $\|X(t,s)\|$ and 
\begin{align*}
\left\| \int_{t_0}^t X(t,s) [ (I-S)^{-1}f](s)\, ds \right\| 
\leq \frac{M_0}{\lambda\left(1-\|A\|_{[t_0,\infty)}\right)} \|f\|_{[t_0,t]}.
\end{align*}
\end{proof}

\begin{theorem}\label{theorem3.1a}
Let $\lambda$ and $\beta$ be positive constants for which
\begin{equation*}
\mu (P(t))\leq -\beta, \quad t \geq t_0, \quad  (1+\lambda  \sigma)e^{\lambda\sigma}\|A\|_{[t_0,\infty)}
+\sum_{k=1}^m \tau_k e^{\lambda\tau_k}\|B_k\|_{[t_0,\infty)}<1,
\end{equation*}
\begin{equation*}
\begin{array}{ll}	
M_2:= & 
\frac{\lambda +\sum\limits_{k=1}^m e^{\lambda\tau_k}\|B_k\|_{[t_0,\infty)}+\lambda e^{\lambda\sigma}\|A\|_{[t_0,\infty)}}
{1-(1+\lambda  \sigma)e^{\lambda\sigma}\|A\|_{[t_0,\infty)}
-\sum_{k=1}^m \tau_k e^{\lambda\tau_k}\|B_k\|_{[t_0,\infty)}} 
\vspace{2mm} \\ 
& \displaystyle  \times \left(\left\|\frac{A}{\mu (P)}\right\|_{[t_0,\infty)}  \!\! \!\!\!\! (1+\lambda\sigma)e^{\lambda\sigma}
+\sum_{k=1}^m  \left\|\frac{ B_k}{\mu (P)}\right\|_{[t_0,\infty)} \!\!\!\!\! e^{\lambda\tau_k}\tau_k \right)<1,
\end{array}
\end{equation*}
where $P$ is defined in \eqref{3.0}.
Then the solution $x$ of (\ref{2.1}),(\ref{2.2}) satisfies \eqref{3.3} with $M_0:=(1-M_2)^{-1}$. 
\end{theorem}
\begin{proof}
The proof of the theorem is similar to the proof of Theorem \ref{theorem3.1} 
if we use Lemma~\ref{lemma2.2} rather than Lemma~\ref{lemma2.1}, (\ref{3.4}) and the same notation.
\end{proof}

Let $m=1$, consider
\begin{equation}\label{3.5}
\dot{x}(t) - A(t)\dot{x}(g(t)) =B(t)x(h(t)) + f(t)
\end{equation}
with initial conditions \eqref{2.2}.
Denote 
\begin{equation*}
P_1(t):= e^{\lambda(t-h(t))}B(t)-\lambda e^{\lambda(t-g(t))} A(t)+\lambda E.
\end{equation*}

\begin{corollary}\label{corollary3.1}
Let $\lambda$ and $\beta$ be positive constants for which
\begin{equation}\label{3.6}
\mu (P_1(t))\leq -\beta, \quad t \geq t_0,\quad  e^{\lambda\sigma}\|A\|_{[t_0,\infty)}<1,
\end{equation}
\begin{equation}
\label{3.7}
M_3:=
\frac{\lambda +e^{\lambda\tau}\|B\|_{[t_0,\infty)}+\lambda e^{\lambda\sigma}\|A\|_{[t_0,\infty)}}{1-e^{\lambda\sigma}\|A\|_{[t_0,\infty)}} 
\left(\left\|\frac{A}{\mu (P_1)}\right\|_{[t_0,\infty)} \!\!\! \!\!\! \!\!\!\! (1+\lambda\sigma)e^{\lambda\sigma}
+ \left\|\frac{ B}{\mu (P_1)}\right\|_{[t_0,\infty)} \!\!\!\!\!\!\!\!\!\! e^{\lambda\tau}\tau \right)<1.
\end{equation}
Then for the solution $x$ of problem (\ref{3.5}),(\ref{2.2}) and $M_0:=(1-M_3)^{-1}$, we have 
\begin{equation}
\label{3.8}
\begin{array}{ll} \|x(t)\|\leq & \displaystyle M_0 e^{-\lambda(t-t_0)}\left[\|x(t_0)\|
+\frac{e^{\lambda \sigma}-1}{\lambda(1-\|A\|_{[t_0,\infty)})} \|A\|_{[t_0,\infty)}\| \Psi \|_{[t_0-\sigma,t_0]}\right.
\vspace{2mm}
\\
& \displaystyle \left.+\frac{e^{\lambda \tau}-1}{\lambda(1-\|A\|_{[t_0,\infty)})} 
\|B\|_{[t_0,\infty)}\|\Phi\|_{[t_0-\tau,t_0]}\right]+\frac{M_0}{\lambda(1-\|A\|_{[t_0,\infty)})}\|f\|_{[t_0,t]}. \end{array}
\end{equation}
\end{corollary}

Consider a delay system
\begin{equation}\label{3.9}
\dot{x}(t)=\sum_{k=1}^m B_k(t)x(h_k(t))+f(t), ~t\geq t_0; ~~x(t)=\Phi(t), ~t\leq t_0,
\end{equation}
which is an initial value problem  for  a particular case of (\ref{1.2}) without the neutral part.

Denote 
\begin{equation*}
\displaystyle P_2(t):=\sum_{k=1}^m e^{\lambda(t-h_k(t))}B_k(t)+\lambda E.
\end{equation*}

\begin{corollary}\label{corollary3.2}
Let $\lambda$ and $\beta$ be positive constants for which
$\mu (P_2(t))\leq -\beta$ for $t \geq t_0$ and
\begin{equation*}
M_4:=
\left(\lambda +\sum\limits_{k=1}^m e^{\lambda\tau_k}\|B_k\|_{[t_0,\infty)}\right) 
\sum_{k=1}^m  \left\|\frac{ B_k}{\mu (P_2)}\right\|_{[t_0,\infty)} \!\!\!\!\!\!\!\!\!\! e^{\lambda\tau_k}\tau_k <1.
\end{equation*}
Then, for $M_0:=(1-M_4)^{-1}$, the solution $x$ of problem (\ref{3.9}) satisfies
\begin{equation*}
 \|x(t)\|\leq \displaystyle M_0e^{-\lambda(t-t_0)}\left[\|x(t_0)\|
+\sum_{k=1}^m\frac{e^{\lambda \tau_k}-1}{\lambda} \|B_k\|_{[t_0,\infty)}\|\Phi\|_{[t_0-\tau_k,t_0]}\right]
+\frac{M_0}{\lambda}\|f\|_{[t_0,t]}.
\end{equation*}
\end{corollary}

Theorem \ref{theorem3.1a} leads to exponential estimates similar to those in Corollaries~\ref{corollary3.1} and \ref{corollary3.2}.

\subsection{Exponential stability}

Theorems \ref{theorem3.1} and \ref{theorem3.1a} immediately yield 
exponential stability conditions.

\begin{theorem}\label{theorem3.2}
Let $\beta$ be a positive constant for which $\displaystyle B(t):=\sum_{k=1}^m B_k(t), ~\mu (B(t))\leq -\beta,~\|A\|_{[t_0,\infty)}<1$ and
\begin{equation}\label{3.12a}
\frac{\sum_{k=1}^m \|B_k\|_{[t_0,\infty)}}{1-\|A\|_{[t_0,\infty)}} 
\left(\left\|\frac{A}{\mu (B)}\right\|_{[t_0,\infty)}+\sum_{k=1}^m \tau_k\left\|\frac{B_k}{\mu (B)}\right\|_{[t_0,\infty)}\right)<1.
\end{equation}
Then equation (\ref{1.2}) is uniformly exponentially stable.
\end{theorem}

\begin{theorem}\label{theorem3.2a}
Let $\beta$ be a positive constant for which $\mu (B(t))\leq - \beta$, where
$$
\displaystyle B(t):=\sum_{k=1}^m B_k(t), ~
~\|A\|_{[t_0,\infty)}+\sum_{k=1}^m \tau_k\|B_k\|_{[t_0,\infty)}<1
$$
and
\begin{equation*}
\left(\frac{\|B\|_{[t_0,\infty)}}{1-\|A\|_{[t_0,\infty)}-\sum_{k=1}^m \tau_k\|B_k\|_{[t_0,\infty)}} \right)
\left(\left\|\frac{A}{\mu (B)}\right\|_{[t_0,\infty)}+\sum_{k=1}^m \tau_k\left\|\frac{B_k}{\mu (B)}\right\|_{[t_0,\infty)}\right)<1.
\end{equation*} 
Then equation (\ref{1.2}) is uniformly exponentially stable.
\end{theorem}

For $C=\{c_{ij}\}_{i,j=1}^n$, denote the matrix $|C|:=\{|c_{ij}|\}_{i,j=1}^n$.
If $|B(t)|\leq \overline{B}$ for any $t \in [t_0,\infty)$, where $\overline{B}$ is a constant matrix, then $\|B \|_{[t_0,\infty)} \leq \| \overline{B} \|$.

\begin{corollary}\label{corollary3.3a}
Let $\beta$ be a positive constant for which $\mu (B(t))\leq - \beta$, where $\displaystyle B(t):=\sum_{k=1}^m B_k(t)$,
$|A(t)|\leq \overline{A}, |B_k(t)|\leq \overline{B_k}, |B(t)|\leq \overline{B}, \|\overline{A}\|<1$, and at least one of the following conditions holds:
$$
\sum_{k=1}^m \|\overline{B_k}\|\left(\|\overline{A}\|+\sum_{k=1}^m \tau_k\|\overline{B_k}\|\right)<\beta (1-\|\overline{A}\|);
$$
$$
  \|\overline{B}\|\left(\|\overline{A}\|+\sum_{k=1}^m \tau_k\|\overline{B_k}\|\right)
	<\beta\left(1-\|\overline{A}\|-\sum_{k=1}^m \tau_k\|\overline{B_k}\|\right).
$$
Then equation (\ref{1.2}) is uniformly exponentially stable.
\end{corollary}

\begin{corollary}\label{corollary3.3}
Let $\beta$ be a positive constant for which $\mu (B(t))\leq - \beta$ for any $t \geq t_0$, $\|A\|_{[t_0,\infty)}<1$, 
$$
\left(\left\|\frac{A}{\mu (B)}\right\|_{[t_0,\infty)}+\tau\left\|\frac{B}{\mu (B)}\right\|_{[t_0,\infty)}\right)\|B\|_{[t_0,\infty)}<1-\|A\|_{[t_0,\infty)}.
$$
Then equation (\ref{3.5}) with $f\equiv 0$ is uniformly exponentially stable.
\end{corollary}

\begin{corollary}\label{corollary3.4}
Let $\beta$ be a positive constant for which $\mu (B(t))\leq - \beta$ for any $t \geq t_0$, where $\displaystyle B(t):=\sum_{k=1}^m B_k(t)$, and either
\\
\vspace{2mm}
1) $\displaystyle 
\sum_{k=1}^m \|B_k\|_{[t_0,\infty)} 
\sum_{k=1}^m \tau_k\left\|\frac{B_k}{\mu (B)}\right\|_{[t_0,\infty)}
<1$
\\
\vspace{2mm}
or
\\
\vspace{2mm}
2) $\displaystyle
\|B\|_{[t_0,\infty)} 
\sum_{k=1}^m \tau_k\left\|\frac{B_k}{\mu (B)}\right\|_{[t_0,\infty)}  
< 1- 
\sum_{k=1}^m \tau_k\left\|B_k\right\|_{[t_0,\infty)}$ 
\\
\vspace{2mm}
is satisfied.
Then, system (\ref{3.9}) with $f\equiv 0$ is uniformly exponentially stable.
\end{corollary}

Consider the system
\begin{equation}\label{3.13}
\dot{x}(t)=B(t)x(h(t)).
\end{equation}

\begin{corollary}\label{corollary3.5}
Let $\beta$ be a positive constant for which $\mu (B(t))\leq - \beta$ for any $t \geq t_0$, and
$$\displaystyle 
\tau\left\|\frac{B}{\mu (B)}\right\|_{[t_0,\infty)}  \!\!\!\! \|B\|_{[t_0,\infty)}<1.
$$
Then system (\ref{3.13}) is uniformly exponentially stable.
\end{corollary}

\section{Examples and Discussion}

First, let us notice that Corollary \ref{corollary3.5} extends the stability test of Proposition~\ref{p3} to a system with a variable delay.

Next, to compare stability tests of the present paper with known ones, consider a linear scalar
neutral differential equation 
\begin{equation}\label{4.10}
\dot{x}(t) - a(t)\dot{x}(g(t)) =\sum_{k=1}^m b_k(t) x(h_k(t)), ~t\geq t_0\geq 0,
\end{equation}
where  $|a(t)|\leq a_0<1$, $t-g(t)\leq \sigma$, $t-h_k(t)\leq \tau_k$. 
Theorems \ref{theorem3.2} and \ref{theorem3.2a} imply the following stability tests for scalar equation \eqref{4.10}.

\begin{corollary}\label{corollary4.10}
Assume that for some $\beta>0,$ $\displaystyle b(t):= \sum_{k=1}^m b_k(t) \leq -\beta <0,~t\geq t_0$ and either 
\begin{equation}\label{4.11}
\sum_{k=1}^m\|b_k\|_{[t_0,\infty)}\left(\left\|\frac{a}{b}\right\|+\sum_{k=1}^m \tau_k \left\|\frac{b_k}{b}  \right\|\right)<1-\|a\|_{[t_0,\infty)}
\end{equation}
or
\begin{equation}\label{4.12}
\|b\|_{[t_0,\infty)}\left(\left\|\frac{a}{b}\right\|
+\sum_{k=1}^m \tau_k \left\|\frac{b_k}{b}\right\|\right)<1-\|a\|_{[t_0,\infty)}-\sum_{k=1}^m \tau_k \|b_k\|.
\end{equation}
Then equation \eqref{4.10} is uniformly exponentially stable.
\end{corollary}

Condition \eqref{4.11}
is known and coincides with 
\cite[Corollary 5.4]{BB1}. 
However, the stability test in \eqref{4.12} is new, to the best of our knowledge, and independent of \eqref{4.11} which is
illustrated in the following example. 

\begin{example}\label{example4.10} 
Consider a scalar equation with variable oscillating coefficients
\begin{equation}\label{4.13}
\dot{x}(t)-\nu\left[0.1 \sin t ~\dot{x}(g(t))=-(1-3\cos t) x(t)-(1+3\cos t) x(h(t))\right],
\end{equation}
where $\nu>0, t-g(t)\leq \sigma, t-h(t)\leq 1=\tau$.
Here $m=2$ and 
$$
\|a\|_{[t_0,\infty)}=0.1 \nu,~\|b_1\|_{[t_0,\infty)}=\|b_2\|_{[t_0,\infty)}=4\nu,~b=b_1+b_2=-2\nu<0, \tau=1.
$$

By \eqref{4.11} equation \eqref{4.13} is uniformly exponentially stable if $\nu<\frac{1}{16.5}$, by \eqref{4.12} the stability condition is 
$\nu<\frac{1}{8.2}$. Hence for this equation condition \eqref{4.12} is better than \eqref{4.11}.

In equation \eqref{4.13}, the value of $\|b_1+b_2\|_{[t_0,\infty)}$ is significantly less than $\|b_1\|_{[t_0,\infty)}+\|b_2\|_{[t_0,\infty)}$.
If they are close, condition \eqref{4.11} is better than \eqref{4.12}. 
For example,  if the term of $3\cos t$ is omitted in the coefficients in 
\eqref{4.13}, estimate in \eqref{4.11} becomes sharper than in \eqref{4.12} ($\nu<\frac{1}{2.2}$ compared to  $\nu<\frac{1}{3.2}$, respectively).
\end{example}

Note that all coefficients in \eqref{4.13} of Example~\ref{example4.10} are oscillating, while most known stability tests even 
for equations without the neutral part deal with positive coefficients.

To compare the results of the present paper with known stability tests,
we adapt Theorems \ref{theorem3.2} and \ref{theorem3.2a} 
to a non-autonomous  equation generalizing \eqref{1.2a} 
\begin{equation}\label{1.2b}
\dot{x}(t) - A_1(t) \dot{x}(H_1(t)) =A_0(t) x(t)+A_2(t) x(H_2(t)), ~t \geq t_0,
\end{equation}
where $0\leq t-H_1(t)\leq h_1, 0\leq t-H_2(t)\leq h_2$ for some $h_1, h_2>0$.

\begin{corollary}\label{corollary4.1}
Assume that for some $\beta>0$, $\mu(A_0(t)+A_2(t))\leq -\beta<0, ~t\geq t_0$  
and at least one of the following  conditions holds:

\begin{equation}\label{4.1}
\displaystyle
\|A_1\|_{[t_0,\infty)}<1,~ -\beta+\frac{(\|A_0\|_{[t_0,\infty)}
+\|A_2\|_{[t_0,\infty)})(\|A_1\|_{[t_0,\infty)}+h_2\|A_2\|_{[t_0,\infty)})}{1-\|A_1\|_{[t_0,\infty)}}<0;
\end{equation}

\begin{equation}\label{4.2}
\displaystyle
\|A_1\|_{[t_0,\infty)}+h_2\|A_2\|_{[t_0,\infty)}<1,~ 
-\beta+\frac{\|A_0+A_2\|_{[t_0,\infty)}(\|A_1\|_{[t_0,\infty)}
+h_2\|A_2\|_{[t_0,\infty)})}{1-\|A_1\|_{[t_0,\infty)}-h_2\|A_2\|_{[t_0,\infty)}}<0.
\end{equation}
Then equation \eqref{1.2b} is uniformly exponentially stable. 
\end{corollary}
\begin{proof}
Assume first that \eqref{4.1} holds. We will apply Theorem \ref{theorem3.2} for $m=2$ to equation \eqref{1.2b}.
 Inequality \eqref{3.12a} for \eqref{1.2b} has the form
 $$
\frac{\|A_0\|_{[t_0,\infty)}+\|A_2\|_{[t_0,\infty)}}{1-\|A_1\|_{[t_0,\infty)}}
\left(\left\|\frac{A_1}{\mu(A_0+A_2)}\right\|_{[t_0,\infty)}
+h_2\left\|\frac{A_2}{\mu(A_0+A_2)}\right\|_{[t_0,\infty)}\right)<1,
$$
which, due to $|\mu(A_0(t)+A_2(t))| \geq \beta$, yields that
\begin{equation}\label{4.3}
\frac{\|A_0\|_{[t_0,\infty)}+\|A_2\|_{[t_0,\infty)}}{1-\|A_1\|_{[t_0,\infty)}}
\left(\|A_1\|_{[t_0,\infty)}+h_2\|A_2\|_{[t_0,\infty)}\right)<\beta.
\end{equation}
Inequality \eqref{4.3} is equivalent to the second inequality in \eqref{4.1}. 

All conditions of Theorem \ref{theorem3.2} hold for equation \eqref{1.2b}, hence this equation is uniformly exponentially stable.

The second part of the corollary follows from Theorem \ref{theorem3.2a} and is justified in a similar way. 
\end{proof}

Proposition \ref{p1} is independent of Proposition \ref{p2} and of Corollary \ref{corollary4.1}. Condition~\eqref{4.1} in Corollary~\ref{corollary4.1} is independent of both propositions  and of the stability test in \eqref{4.2}, while 
\eqref{4.2} coincides with Proposition \ref{p2}
in the case of constant coefficients and delays.

Several stability results for linear neutral systems were obtained by the LMI method based on Lyapunov-Krasovskii functionals.
Usually  neutral systems  studied by the LMI method have constant matrix coefficients and bounded derivatives of delay functions. 

Consider now the following example of a non-autonomous system of neutral type with variable matrix coefficients $n\times n$.
Note that with the growth of dimension the calculations in LMI increase dramatically.  

\begin{example}\label{example2}
Consider system \eqref{1.2} for $m=2$
\begin{equation}\label{1.2A}
\dot{x}(t) -  A(t)\dot{x}(g(t)) = B_1(t)x(h_1(t))+B_2(t)x(h_2(t)), ~~t \geq t_0\geq 0,
\end{equation}
where $B_1(t)=\sin^2 t~ C$, $B_2(t)=\cos^2 t ~C$,  $C$ is a tri-diagonal and $A$ is a variable matrix
$$
C=\begin{pmatrix}
-\alpha&\beta&0&0&\dots&0&0\\
\beta/2 &-\alpha&\beta/2&0&\dots&0&0\\
0&\beta/2 &-\alpha&\beta/2&\dots&0&0\\
\vdots&\vdots&\vdots&\vdots&\vdots&\vdots&\vdots\\
0&0 &0&0 &\dots &\beta&-\alpha
\end{pmatrix},
~~A=\gamma\begin{pmatrix}
\cos t& \cos 2t&\cos 3t&\dots&\cos nt\\
\cos^2 t& \cos^2 2t&\cos^2 3t&\dots&\cos^2 nt\\
\vdots&\vdots&\vdots&\vdots&\vdots\\
\cos^n t& \cos^n 2t&\cos^n 3t&\dots&\cos^n nt\\
\end{pmatrix}, 
$$
$\alpha>0, h_1(t)=t-0.1|\sin t|, h_2(t)=t-0.1|\cos t|, t-g(t)\leq \sigma$.
We have for the norm $\|\cdot \|_{[0,\infty)}$:
$$
\|A\|_{[0,\infty)}=n|\gamma|, ~~\tau_i=\sup_{t\geq 0} (t-h_i(t))=0.1, ~~\|B_i\|_{[0,\infty)}=\alpha+|\beta|, ~i=1,2,$$
$B=B_1+B_2$, $\mu (B)=-\alpha+|\beta|$.
Let
\begin{equation}\label{4.1a}
|\beta|<\alpha,  ~ n|\gamma|<1, ~
2(\alpha+|\beta|)[n|\gamma|  +0.2(\alpha+|\beta|)]<(1-n|\gamma|)(\alpha-|\beta|).
\end{equation}
Then, by Theorem \ref{theorem3.2} equation \eqref{1.2A}  
is uniformly exponentially stable.

In particular,  conditions \eqref{4.1a} hold for $\alpha=0.4$, $|\beta|=0.1$, $n|\gamma|=0.01$.

Assume 
in addition $\|A\|=n|\gamma|=0.01$, $\lambda=0.06$, $\sigma=0.1$. Hence $\|B_i\|_{[0,\infty)}=0.5$, $\mu(C)=-0.3$
and calculate by Theorem \ref{theorem3.1}  
an estimate for solution $x$ of the equation
\begin{equation}\label{1.2B}
\dot{x}(t) - A(t)\dot{x}(g(t)) = B_1(t)x(h_1(t))+B_2(t)x(h_2(t))+f(t), ~~t \geq t_0\geq 0,
\end{equation}
with initial condition \eqref{2.2}. We have 
\begin{align*}
\mu (P(t))= & \mu\left(e^{\lambda (t-h_1(t))}B_1(t)+e^{\lambda (t-h_2(t))}B_2(t)-\lambda e^{\lambda (t-g(t))} A(t)+\lambda E\right)
\\
& \leq \mu\left(e^{\lambda (t-h_1(t))}B_1(t)+e^{\lambda (t-h_2(t))}B_2(t)\right)+\mu\left(-\lambda e^{\lambda (t-g(t))} A(t)\right)+\mu(\lambda E).
\end{align*}
Then
$$
\mu\left(e^{\lambda (t-h_1(t))}B_1(t)+e^{\lambda (t-h_2(t))}B_2(t)\right)=\left(e^{\lambda (t-h_1(t))}\sin^2 t+e^{\lambda (t-h_2(t))}\cos^2 t\right) \mu (C).
$$
Since $\mu (C)=-0.3<0$ and 
$$
e^{\lambda (t-h_1(t))}\sin^2 t+e^{\lambda (t-h_2(t))}\cos^2 t\geq \sin^2 t+\cos^2 t=1,
$$ 
we get
$$
\mu\left(e^{\lambda (t-h_1(t))}B_1(t)+e^{\lambda (t-h_2(t))}B_2(t)\right)\leq \mu(C)=-0.3.
$$
Next,
$$
\mu\left(-\lambda e^{\lambda (t-g(t))} A(t)\right)\leq \lambda e^{\lambda (t-g(t))} \|A\|_{[0,\infty)}\leq 0.0006 e^{0.006},~
\mu(\lambda E)=\lambda =0.06.
$$
Then $\mu (P(t)) \leq -0.23939<0$ and
$$
M_1:=
\frac{\lambda +\sum\limits_{k=1}^2 e^{\lambda\tau_k}\|B_k\|_{[t_0,\infty)}+
\lambda e^{\lambda\sigma}\|A\|_{[t_0,\infty)}}{1-e^{\lambda\sigma}\|A\|_{[t_0,\infty)}}
\left(\left\|\frac{A}{\mu (P)}\right\|_{[t_0,\infty)} \!\!\! \!\!\! \!\!\!\! (1+\lambda\sigma)e^{\lambda\sigma}
+\sum_{k=1}^2  \left\|\frac{ B_k}{\mu (P)}\right\|_{[t_0,\infty)} \!\!\!\!\!\!\!\!\!\! e^{\lambda\tau_k}\tau_k \right).
$$
By numerical calculations, $M_1\leq 0.5$ and  $M_0=(1-M_1)^{-1}\leq 2$. 
For the solution $x$ of problem \eqref{1.2B},\eqref{2.2} we have (here we omitted the norm indices)
$$
\|x(t)\|\leq \left. \left. 2 e^{-0.06(t-t_0)}\right[ \|x(t_0)\|+0.00102 \|\Psi\| +0.102\|\Phi\|\right]+33.6 \|f\|, ~t\geq t_0\geq 0.
$$
\end{example}

Let us conclude with some additional comments. Several explicit stability conditions for delay and neutral systems were obtained by application of the Bohl-Perron theorem \cite{ AzbSim,BDSZ1,BDSZ2,BDSZ3,Dom,Gil}. 
In general, these tests do not coincide with results of this paper even in the scalar case, see discussion in \cite{BB2}.
However, some conditions can be the same, for example,
Corollary~\ref{corollary3.4} of the present paper and \cite[Corollary 1]{BDSZ3} coincide.

In interesting papers \cite{Ngoc2,Ngoc1}, the authors study neutral systems in the Hale form. 
They obtained explicit exponential stability conditions 
for rather general linear neutral systems. The method applied in \cite{Ngoc2,Ngoc1} is based 
on some specific spectral matrix properties, such as a Metzler matrix with non-negative off-diagonal entries.

Fixed point methods also lead to explicit stability  condition, but so far they were  applied only to scalar neutral equations,
and used the first and the second derivatives of delay functions. The advantage of fixed point approach is that it can also be applied to nonlinear equations. The method described here is only considered for linear systems.

Finally, we suggest some relevant topics for future research.

\begin{enumerate}
\item
Here we used the essential supremum norm which led to exponential stability tests including essential supremum estimates.
It would be interesting to get integral stability conditions for neutral systems considered in the paper, 
for the scalar case see \cite[Theorem 4.7]{BB1}.
\item
Obtain stability conditions assuming that either delays or coefficients, or both, can be unbounded.
\item
Investigate exponential stability for neutral systems of a higher order. 
\item
Are uniform asymptotic stability and uniform exponential stability equivalent for linear neutral systems with
bounded delays? For equations without a neutral term this is known \cite{Hale, Malygina}. 
\end{enumerate}

\section*{Acknowledgment}

E. Braverman  was partially supported by NSERC, the grant RGPIN-2020-03934.
The authors are very grateful to the anonymous referees whose thoughtful comments significantly contributed to the quality of presentation.

\end{document}